\documentclass[11pt,a4paper]{amsart}

\usepackage{todonotes}

\usepackage[shortlabels]{enumitem}
\usepackage{yfonts}
\usepackage{mathtools}
\usepackage{amsfonts}
\usepackage{mathabx} 
\usepackage{relsize} 
\usepackage{bbm} 
\usepackage{amssymb, amsthm, array}
\usepackage{amsmath}
\usepackage{hyperref}
\allowdisplaybreaks

\usepackage{breqn}

\usepackage{geometry} 

\usepackage{tabu}

\usepackage{tikz-cd}
\usepackage{multicol}

\usepackage{extarrows}

\usepackage{mathdots}

\usepackage{tikz}
\usetikzlibrary{arrows}
\usetikzlibrary{decorations.markings}

\tikzset{
    o/.style={
        shorten >=#1,
        decoration={
            markings,
            mark={
                at position 1
                with {
                    \draw circle [radius=#1];
                }
            }
        },
        postaction=decorate
    },
    o/.default=3pt
}

\tikzset{SmallNode/.style={draw,inner sep=0pt},
every edge quotes/.style={fill=white}
}

\allowdisplaybreaks

\DeclareMathAlphabet\mathbfcal{OMS}{cmsy}{b}{n}

\DeclareMathOperator{\Ext}{Ext} 

\let\epsilon\varepsilon

\newcommand{\N}{\mathbb{N}} 
\newcommand{\Z}{\mathbb{Z}} 
\newcommand{\R}{\mathbb{R}} 
\newcommand{\C}{\mathbb{C}} 

\newcommand{\cK}{\mathcal{K}} 
\newcommand{\cL}{\mathcal{L}} 
\newcommand{\cI}{\mathcal{I}} 

\newcommand{\cZ}{\mathcal{Z}} 
\newcommand{\ZK}{{\cZ}_{\cK}} 

\newcommand{\D}{\mathbb{D}} 

\newcommand{\biggast}{\mathlarger{\mathlarger{\mathlarger{\bigast}}}} 

\newtheorem{theorem}{Theorem}[section]
\newtheorem*{theorem*}{Theorem}
\newtheorem{proposition}[theorem]{Proposition}

 \newtheorem{corollary}[theorem]{Corollary}
 
\theoremstyle{definition}
\newtheorem{definition}[theorem]{Definition}
\newtheorem{example}[theorem]{Example}

\newtheoremstyle{break}
  {\topsep}{\topsep}%
  {\upshape}{}%
  {\bfseries}{}%
\theoremstyle{break}

\newtheoremstyle{defbreak}%
  {\topsep}{\topsep}%
  {\upshape}{}%
  {\bfseries}{}%
\theoremstyle{defbreak}

\newtheoremstyle{examplebreak}%
  {\topsep}{\topsep}%
  {\upshape}{}%
  {\bfseries}{}%
\theoremstyle{examplebreak}

\newtheoremstyle{examplesbreak}%
  {\topsep}{\topsep}%
  {\upshape}{}%
  {\bfseries}{}%
\theoremstyle{examplesbreak}

\newtheoremstyle{remarkbreak}%
  {\topsep}{\topsep}%
  {\upshape}{}%
  {\bfseries}{}%
\theoremstyle{remarkbreak}

\newtheoremstyle{remarksbreak}%
  {\topsep}{\topsep}%
  {\upshape}{}%
  {\bfseries}{}%
\theoremstyle{remarksbreak}

\newtheoremstyle{lemmabreak}%
  {\topsep}{\topsep}%
  {\upshape}{}%
  {\bfseries}{}%
\theoremstyle{lemmabreak}

\newtheoremstyle{propbreak}%
  {\topsep}{\topsep}%
  {\upshape}{}%
  {\bfseries}{}%
\theoremstyle{propbreak}

\newtheoremstyle{corrbreak}%
  {\topsep}{\topsep}%
  {\upshape}{}%
  {\bfseries}{}%
\theoremstyle{corrbreak}

\newtheoremstyle{thmbreak}%
  {\topsep}{\topsep}%
  {\upshape}{}%
  {\bfseries}{}%
\theoremstyle{thmbreak}

\newtheoremstyle{conjbreak}%
  {\topsep}{\topsep}%
  {\upshape}{}%
  {\bfseries}{}%
\theoremstyle{conjbreak}

\newtheoremstyle{constrbreak}%
  {\topsep}{\topsep}%
  {\upshape}{}%
  {\bfseries}{}%
\theoremstyle{constrbreak}

\numberwithin{equation}{section}

\title{Duality in Toric Topology}
\author{Jelena Grbi\'c}
\author{Matthew Staniforth}
\address{School of Mathematical Sciences, University of Southampton, Southampton, SO171BJ, United Kingdom \\ \textit{Email address}: \href{mailto:j.grbic@soton.ac.uk}{j.grbic@soton.ac.uk}}
\address{School of Mathematical Sciences, University of Southampton, Southampton, SO171BJ, United Kingdom \\ \textit{Email address}: \href{mailto:j.grbic@soton.ac.uk}{m.staniforth@soton.ac.uk}}

\begin{document}

\maketitle

\begin{abstract}
We characterise integral Poincar\'e duality moment-angle complexes $\ZK$ in combinatorial terms of the Alexander duality of the simplicial complex $\cK$, and consequently in algebraic terms of the Gorenstein duality of the Stanley-Reisner ring $\Z[\cK]$.  We extend Poincar\' e duality results to certain polyhedral products using  polyhedral join products of simplicial complexes.

\end{abstract}

\section{Introduction}
The polyhedral product $(\mathbf{X},\mathbf{A})^{\cK}$ of topological pairs $(X_i,A_i)$ is a subspace of the cartesian product $\prod X_i$ which is specified by the face category of a simplicial complex $\cK$. The homotopy theory of polyhedral product spaces is a rapidly evolving area of algebraic topology, and the tools of homotopy theory can often be enhanced using both algebraic and combinatorial techniques when being brought to bear on the study of polyhedral products.

A polyhedral product of particular interest in Toric Topology is the moment-angle complex, where $(X_i,A_i) = (D^2,S^1)$, which comes readily equipped with the action of a torus. A study of moment-angle complexes and related polyhedral products not only allows us to gain insight into these spaces themselves, but also provides us with a framework within which we can investigate an interplay of homotopy theoretic, algebraic and combinatorial phenomena. In this paper, we investigate the interaction of duality phenomena in these areas. 

An integral Poincar\'e duality space $X$ is one whose action of its integral cohomology algebra on its integral homology satisfies Poincar\'e duality, that is, there exists $n \in \N$ and $[\mu] \in H_n(X)$ such that the cap product
\[
[\mu] \frown (-)\colon H^l(X) \longrightarrow H_{n-l}(X) 
\] 
is an isomorphism for all $l$. Any manifold satisfies Poincar\'e duality, but not every Poincar\'e duality space is a manifold. For example, an integral homology $n$-manifold, a space with the same integral local homology groups as $\R^n$, satisfies Poincar\'e duality. We characterise Poincar\'e duality moment-angle complexes $\ZK$ in terms of a duality of the underlying simplicial complex $\cK$.

A generalised homology $n$-sphere ($GHS^n$) is a homology $n$-manifold with the homology of $S^n$. In 2015, Fan and Wang \cite{fan2016cohomology} characterised the simplicial complexes $\cK$ for which the geometric realisation $|\cK|$ is a $GHS^n$ in terms of a duality condition on the homology and cohomology groups of full subcomplexes of $\cK$, which we refer to as combinatorial Alexander duality. In Theorem~\ref{thm:ZKisPDiffKGHS}, we show that the moment-angle complex $\ZK$ being a Poincar\'e duality space is equivalent to the condition that $\cK$ exhibits combinatorial Alexander duality. As a corollary of our result and Cai's~\cite{cai2015products} classification of moment-angle manifolds,  we obtain that there are no Poincar\'e duality moment-angle complexes which are not manifolds.

Duality phenomena are ubiquitous across mathematics, and are not limited to topology and combinatorics.  A duality which appears in commutative algebra is Gorenstein duality; a property for a $d$-dimensional Noetherian ring $R$ that is measured by the functor $\Ext^{d-t}(-,R)$. The Avramov-Golod Theorem \cite[Theorem~3.4.5]{bruns_herzog_1998}, equivocates Gorenstein duality of the Stanley-Reisner ring $\mathbbm{k}[\cK]$, where $\mathbbm{k}$ is a field, with Poincar\'e of its Tor-algebra.
The framework of Toric Topology allows us to investigate an interplay of Gorenstein duality in the integral Stanley-Reisner ring $\Z[\cK]$ with topological and combinatorial dualities, in the simplicial complex $\cK$ and the moment-angle complex $\ZK$ respectively. Paraphrasing Stanley's result~\cite{alma991027115349703276} on Gorenstein Stanley-Reisner rings $\Z[\cK]$, in Theorem~\ref{thm:ZKPDiffKFWiffZ[K]Gor} we show that Poincar\'e duality in $\ZK$ is equivalent to Gorenstein duality in $\Z[\cK]$, realising an interplay of algebraic, combinatorial and topological dualities. The complements the result proven by Buchstaber and Panov \cite[Theorem~4.6.8]{buchstaber2014toric} in the case of coefficients over a field.

As a cartesian product of simplicial complexes is not a simplicial complex, a polyhedral product $(\mathbfcal{K}_{\mathbf{i}},\mathbfcal{L}_{\mathbf{i}})^{\cK}$ of simplicial pairs $(\cK_i,\cL_i)$ is not a simplicial complex. A related notion to the polyhedral product exists, where a simplical complex, known as the polyhedral join product~\cite{Vidaurre}, is  constructed as a union of join products of simplicial complexes. The special cases known as substitution complexes and composition complexes were studied by Abramyan-Panov \cite{MR4017598}  and Ayzenberg \cite{anton2013composition}, respectively. In Theorem~\ref{thm:PDFanWangGorensteinrelationshipfullgenerality}, using the polyhedral join product, we specify a family
of polyhedral products which satisfy Poincar\' e duality.

\goodbreak

\section{Duality of $\ZK$ and $\Z[\cK]$}
\label{sec:dualityinKZKandZ[K]}

\subsection{Preliminaries: The cohomology of $\ZK$}
\label{subsec:prelimcohomofZK}

For a positive integer $m$, a simplicial complex on the vertex set $[m] = \{1,...,m\}$ is a subset of $2^{[m]}$ which is closed under taking subsets, and contains the empty set. We allow a simplicial complex $\cK$ to contain ghost vertices; that is, we allow that there might exist $i \in [m]$ such that $\{i\} \notin \cK$. 

\begin{definition}

Let $\cK$ be a simplicial complex on vertex set $[m]$, and denote by $\mathbf{(X,A)}=\{(X_i,A_i)\}_{i=1}^m$ an $m$-tuple of $CW$-pairs. The \textit{polyhedral product} is defined as

\[
\begin{aligned}
\mathbf{(X,A)}^{\cK}  = \bigcup_{\sigma \in \cK} \mathbf{(X,A)}^\sigma \subseteq \prod_{i=1}^m X_i, \text{ where }
\mathbf{(X,A)}^\sigma   = \prod_{i=1}^m Y_i, \quad  Y_i = \begin{cases} X_i & \text{for }i \in \sigma \\ A_i & \text{for }i \notin \sigma. \end{cases}
\end{aligned}
\]

\end{definition}

If $(X_i,A_i) = (X,A)$ for all $i$, we denote the polyhedral product by $(X,A)^{\cK}$. When $(X_i,A_i) = (D^2,S^1)$ for all $i$, the polyhedral product is denoted by $\ZK$, and referred to as the moment-angle complex on $\cK$.

 We begin with a description of the integral cellular cochain complex $C^*(\ZK;\Z)$ due to Panov and Buchstaber \cite[Section~4.4]{buchstaber2014toric}.

Let $\D^m$ denote the $m$-dimensional unit ball in $\C^m$. The disk $\D^1$ admits a decomposition into $3$ cells: the basepoint $*$, the boundary circle $S$, and the $2$-cell $D$.

Taking products, we obtain a cellular decomposition of $\D^m$. A cell $e$ of $\D^m$ is a product of cells of $\D$ of the form $\prod_{i=1}^m Y_i$,
where for each $i$, $Y_i$ is either the basepoint $*$, the $1$-cell $S$ or the $2$-cell $D$. This can be phrased in terms of subsets of $[m]$. Each cell of $\D^m$ corresponds to exactly one pair $(J,I)$ of subsets $J,I \subseteq [m]$, with $J \cap I = \emptyset$, and this pair characterises unique a cell of $\D^m$. The subset $J$ corresponds to the $1$-cells, $I$ corresponds to the $2$-cells, and $[m] \backslash (J \cup I)$ corresponds to the $0$-cells. We denote the cell corresponding to the pair $(J,I)$ by $\kappa(J,I)$. The dimension of such a cell is $$\dim{\kappa(J,I)}=|J| + 2|I|.$$

This $CW$-structure on $\D^m$ induces a sub $CW$-structure on $\ZK \subseteq \D^m$. For $J \subseteq [m]$, we denote by $\cK_J = \{\sigma \in \cK \; | \; \sigma \subseteq J \}$ the full subcomplex of $\cK$ on $J$. Then, $\kappa(J \backslash \sigma,\sigma) \in \ZK$ if and only if $\sigma \in \cK_J$. 

 We denote by $C_*(\cK)$ and $C_*(\ZK)$ the simplicial and cellular chain complexes of $\cK$ and $\ZK$, respectively. Here and throughout, coefficients are taken to be in $\Z$, and we observe the convention that $C_{-1}(\cK) \cong C^{-1}(\cK) \cong \langle \emptyset_* \rangle = \Z$. There is the isomorphism of graded modules

\[
\begin{aligned}
h:\bigoplus_{J \subseteq [m]} C_*(\cK_J) &\xlongrightarrow{\cong} C_*(\ZK), \quad \sigma \mapsto \kappa(J \backslash \sigma,\sigma)
\end{aligned}
\]
where the grading on the left hand side is given by $\deg{\sigma} = 2|\sigma| + |J|$ for $\sigma \in \cK_J$. 

By theorems of Hochster, Baskakov, Panov and Buchstaber, the map $h$ induces the isomorphism of cohomology rings \cite[Theorem~4.5.7]{buchstaber2014toric}
\begin{equation} \label{eq:hochster}
h^*:H^*(\ZK) \cong \bigoplus_{J \subseteq [m]} \tilde{H}^*(\cK_J)
\end{equation}
where the ring structure on the right hand side is induced by the cochain-level Baskakov product
\[
C^{p-1}(\cK_I) \otimes C^{q-1}(\cK_J) \rightarrow C^{p+q-1}(\cK_{I \cup J}), \quad \sigma^* \otimes \tau^* \mapsto \begin{cases} (\sigma \cup \tau)^* & \text{if } I \cap J = \emptyset \\
0 & \text{otherwise}\\
\end{cases}
\]
where $(\sigma \cup \tau)^*$ is zero if $\sigma \cup \tau \notin \cK_{I \cup J}$, and otherwise denotes the cochain dual to $(\sigma \cup \tau) \in C_*(\cK_{I \cup J})$.
\goodbreak

\subsection{Poincar\'e duality of $\ZK$}
\label{subsec:poincaredualityinZK}

For $d \in \N$, a $CW$-complex $X$ is an $n$-\textit{Poincar\'e duality space} if there exists a class $[\mu] \in H_n(X)$ such that  the cap product
$$[\mu] \frown (-)\colon H^l(X) \rightarrow H_{n-l}(X)$$
is an isomorphism for all $l$. Here, $n$ is referred to as the Poincar\'e duality-dimension of $X$, and $[\mu]$ is referred to as the fundamental class. 

The characterisation of the structure of the cellular chains and cochains of $\ZK$ in terms of the simplicial chains of full sub-complexes of $\cK$ allows us to reframe statements about Poincar\'e duality of $\ZK$ as statements about duality of the simplicial chains and cochains of full subcomplexes of $\cK$. 

We start with a description of the cap product in $\ZK$, on the cellular level, in terms of the combinatorics of the simplicial complex $\cK$. We then exploit this combinatorial description to obtain a characterisation of Poincar\'e duality of $\ZK$ in terms of $\cK$, and also in terms of the Stanley-Reisner ring $\Z[\cK]$. 

\goodbreak

\begin{proposition}
\label{prop:capprodinZK}

Let $\kappa(J \backslash \sigma,\sigma) \in C_*(\ZK)$ and $\kappa(\widehat{J} \backslash \widehat{\sigma},\widehat{\sigma})^* \in C^*(\ZK)$, corresponding to simplices $\sigma \in \cK_J$ and $\widehat{\sigma} \in \cK_{\widehat{J}}$, respectively. Then the cap product is given by

\[
\kappa(J \backslash \sigma,\sigma) \frown \kappa(\widehat{J} \backslash \widehat{\sigma},\widehat{\sigma})^* = \begin{cases}

0 & \widehat{J} \nsubseteq J \\
0 & \widehat{\sigma} \nsubseteq \sigma \\
\kappa \left( (J \backslash \sigma) \backslash (\widehat{J} \backslash \widehat{\sigma}), \sigma \backslash \widehat{\sigma} \right) & \text{otherwise.}\\
\end{cases}
\]

\end{proposition}

\goodbreak

\begin{proof}

For any $CW$-complex, the cellular chain-level cup $\smile$ and cap $\frown$ products satisfy
$$\langle \alpha, \phi \smile \psi \rangle = \langle \alpha \frown \phi,\psi \rangle$$ 
for $\alpha \in C_{k+l}(\ZK), \phi \in C^l(\ZK), \psi \in C^k(\ZK)$, where $\langle - , - \rangle$ denotes the evaluation pairing.\\

Let $\kappa(J \backslash \sigma,\sigma) \in C_*(\ZK)$ and $\kappa(\widehat{J} \backslash \widehat{\sigma},\widehat{\sigma})^* \in C^*(\ZK)$. We write the cap product $\kappa(J \backslash \sigma,\sigma) \frown \kappa(\widehat{J} \backslash \widehat{\sigma},\widehat{\sigma})^*$ in terms of generators $C_*(\ZK)$. For $L \subseteq [m]$ and $\tau \in \cK_L$, the coefficient of a generator $\kappa(L \backslash \tau,\tau) \in C_*(\ZK)$ in $\kappa(J \backslash \sigma,\sigma) \frown \kappa(\widehat{J} \backslash \widehat{\sigma},\widehat{\sigma})^*$ is given by 
\[
\langle (\kappa(J \backslash \sigma,\sigma) \frown \kappa(\widehat{J} \backslash \widehat{\sigma},\widehat{\sigma})^*), \kappa(L \backslash \tau,\tau)^* \rangle = \langle(\kappa(J \backslash \sigma,\sigma), \kappa(\widehat{J} \backslash \widehat{\sigma},\widehat{\sigma})^* \smile \kappa(L \backslash \tau,\tau)^* \rangle.
\]
Now,

\[
\langle(\kappa(J \backslash \sigma,\sigma), \kappa(\widehat{J} \backslash \widehat{\sigma},\widehat{\sigma})^* \smile \kappa(L \backslash \tau,\tau)^* \rangle \neq 0 \\
\]
is equivalent to
\[
(J \backslash \sigma,\sigma) = ((\widehat{J} \backslash \widehat{\sigma}) \cup (L \backslash \tau),\widehat{\sigma} \cup \tau).
\]
Thus $\kappa(J \backslash \sigma,\sigma) \frown \kappa(\widehat{J} \backslash \widehat{\sigma},\widehat{\sigma})^*$ is non-trivial if and only if $\widehat{J} \backslash \widehat{\sigma} \subseteq J \backslash \sigma$ and $\widehat{\sigma} \subseteq \sigma$, whence $\kappa(J \backslash \sigma,\sigma) \frown \kappa(\widehat{J} \backslash \widehat{\sigma},\widehat{\sigma})^* = \kappa((J \backslash \sigma) \backslash (\widehat{J} \backslash \widehat{\sigma}), \sigma \backslash \widehat{\sigma})$.

\end{proof}

We show that $n$-Poincar\'e duality spaces $\ZK$ are characterised by a duality in $\cK$ referred to as combinatorial Alexander duality.  A space $X$ is a $GHS^n$ if it is a homology $n$-manifold with the homology of $S^n$. Fan and Wang \cite[Theorem~3.4]{fan2016cohomology} showed that for $\cK$ a simplicial complex of dimension $n$ on vertex set $[m]$, $\cK$ is a $GHS^{n}$ if and only if
\[\tilde{H}^{l}(\cK_J) \cong \tilde{H}_{n-l-1}(\cK_{[m] \backslash J})\]
for all $J \subseteq [m]$, $0 \leq l \leq n$. In this case we say that $\cK$ has $n$-dimensional combinatorial Alexander duality.

We are now ready to give our combinatorial classification of Poincar\'e duality moment-angle complexes.

\goodbreak

\begin{theorem}
\label{thm:ZKisPDiffKGHS}
Let $\cK$ be a simplicial complex on $[m]$ with non-trivial cohomology. Then $\ZK$ is an $(n+m)$-Poincar\'e duality space if and only if $\cK$ satisfies $(n-1)$-dimensional combinatorial Alexander duality.
\end{theorem}

\begin{proof}

The sufficient implication is settled by a result of Cai \cite[Corollary~2.10]{cai2015products}; if $\cK$ satisfies $(n-1)$-dimensional combinatorial Alexander duality and has non-trivial cohomology, then $\ZK$ is an $(n+m)$-dimensional manifold.

We show the necessary implication. Let $\cK$ be a simplicial complex on $[m]$, with non-trivial cohomology, and suppose that $\ZK$ is an $(n+m)$-Poincar\'e duality space. We show first that $\cK$ has the homology of $S^{n-1}$. We subsequently utilise this fact in showing that a certain chain is a representative of the fundamental class $[\mu] \in H_{n+m}(\ZK)$.

As the simplicial complex $\cK$ has non-trivial cohomology, for some $l$, there exists $0 \neq  [\tau] \in \tilde{H}^l(\cK)$. Let $\tau= \sum_{j} \alpha_{j} \tau_j^*$ , where $\tau_j^* \in C^l(\cK)$ are basis cochains, corresponding to simplices $\tau_j$.  We show that $l$ must equal $n-1$. 

The image of $[\tau]$ under isomorphism \eqref{eq:hochster} is the class
\[
\begin{split}
h^*([\tau])=\left[ \sum_{j} A_{j} \kappa([m] \backslash \tau_j, \tau_j)^* \right] \in H^{l+m+1}(\ZK)
\end{split}
\]
where $A_{j} = \mathrm{sgn}(\tau_j,[m])\alpha_{j}$.

Let $0 \neq [\mu] \in H_{n+m}(\ZK)$ denote the fundamental class, represented by $\mu = \sum_i a_i \kappa(J_i \backslash \sigma_i, \sigma_i)$. We evaluate the product
$$0 \neq [\mu] \frown h^*([\tau]) = \left[ \sum_{i, j}  a_i A_{j} \left( \kappa(J_i \backslash \sigma_i, \sigma_i) \frown \kappa([m] \backslash \tau_j, \tau_j)^* \right) \right] \in  H_{n-(l+1)}(\ZK).$$

By Proposition~\ref{prop:capprodinZK},
\[
\kappa(J_i \backslash \sigma_i, \sigma_i) \frown \kappa([m] \backslash \tau_j, \tau_j)^* \neq 0\] implies that \[\tau_j \subseteq \sigma_i, \text{ and } [m] \backslash \tau_j \subseteq J_i \backslash \sigma_i\] and therefore \[\ J_i = [m] \text{ and } \tau_j = \sigma_i.
\]
Thus, the non-triviality of $\kappa(J_i \backslash \sigma_i, \sigma_i) \frown \kappa([m] \backslash \tau_j, \tau_j)^*$ implies that \[ \kappa(J_i \backslash \sigma_i, \sigma_i) \frown \kappa([m] \backslash \tau_j, \tau_j)^* = \kappa(\emptyset, \emptyset).\]

Therefore
\[
\begin{aligned}
[\mu] \frown h^*([\tau]) &= \left[ \sum_{i, j} \mathrm{sgn}(\tau_j,[m]) a_i \alpha_{\tau_j} \left( \kappa(J_i \backslash \sigma_i, \sigma_i) \frown \kappa([m] \backslash \tau_j, \tau_j)^* \right) \right] \\
&= \left[ A \kappa(\emptyset,\emptyset) \right] \in H_0(\ZK)\\
\end{aligned}
\]
where $A \neq 0$. It follows that $h^*([\tau]) \in H^{n+m}(\ZK)$, so that $[\tau] \in \tilde{H}^{n-1}(\cK)$ by the definition of the isomorphism $h^*$.

We have obtained that the $(n-1)$-st cohomology group of $\cK$ is the only non-trivial cohomology group. It remains to show that $\tilde{H}^{n-1}(\cK) \cong \Z$. By Poincar\'e duality, we have that $H^{n+m}(\ZK) \cong H_0(\ZK) \cong \Z.$ By \eqref{eq:hochster},  $\tilde{H}^{n-1}(\cK)$ includes into $H^{n+m}(\ZK) \cong \Z$ as a component of a direct sum. It follows that $h^*\colon H^{n+m}(\ZK) \rightarrow \tilde{H}^{n-1}(\cK) \cong \Z$ is an isomorphism. Therefore $\cK$ has the homology of $S^{n-1}$, as claimed.
It follows that the fundamental class $[\mu] \in H_{n+m}(\ZK)$ can be represented by a cellular chain of the form $\mu = \sum_i a_i \kappa([m] \backslash \sigma_i, \sigma_i)$.

We now show that $\cK$ has combinatorial Alexander duality, that is, for any $J \subseteq [m]$, and $0 \leq l \leq n+m$, 
\[
\tilde{H}^{l}(\cK_J) \cong \tilde{H}_{n-l-2}(\cK_{[m] \backslash J}).
\]

By Poincar\'e duality, we have the sequence of isomorphisms
\[
\bigoplus_{J \subseteq [m]} \tilde{H}^{\hat{l}-|J| -1} (\cK_J) \cong  H^{\hat{l}}(\ZK) \cong H_{n+m - \hat{l}}(\ZK) \cong \bigoplus_{L \subseteq [m]} \tilde{H}_{(n+m-\hat{l})-|L|-1}(\cK_L)
\]
given by the composition $(h_*)^{-1} \circ ([\mu] \frown (-)) \circ h^*$, where $[\mu] \in H_{n+m}(\ZK)$ denotes the fundamental class of $\ZK$. Substituting $\widehat{l} = l + |J| + 1$,
\[
\bigoplus_{J \subseteq [m]} \tilde{H}^{l} (\cK_J) \cong  H^{l+|J|+1}(\ZK) \cong H_{n+m - (l+|J|+1)}(\ZK) \cong \bigoplus_{L \subseteq [m]} \tilde{H}_{n+m-l-|J|-|L|-2}(\cK_L).
\]

Denoting the composite isomorphism by $\Phi$, we show that $\Phi$ respects the direct sum decompositions. In particular, for all $ J \subseteq [m]$,

\[
\begin{aligned}
\Phi(\tilde{H}^{l}(\cK_J)) &\subseteq \tilde{H}_{n+m-l - |J| - ([m] - |J|) -2}(\cK_{[m] \backslash J}) = \tilde{H}_{n-l-2}(\cK_{[m] \backslash J}).
\end{aligned}
\] 

Suppose that for $J \subseteq [m]$, $\cK_J$ has non-trivial cohomology. Otherwise, the statement follows vacuously. Let $0 \neq [\tau] \in \tilde{H}^{l}(\cK_J)$ with representative cochain $\tau = \sum_j \alpha_j \tau_j^*$. Then
$$h^*([\tau]) = \left[ \sum_j \mathrm{sgn}(\tau_j,J) \alpha_j \kappa(J \backslash \tau_j, \tau_j)^* \right] \in H^{l+|J|+1}(\ZK).$$
Let $[\mu]$ denote the fundamental class of $\ZK$ with representative chain $\mu = \sum_i a_i \kappa([m] \backslash \sigma_i, \sigma_i)$. Evaluating the cap product gives

\[
\begin{aligned}
[\mu] \frown h_c([\tau]) &= \left[ \sum_{i,j} \mathrm{sgn}(\tau_j,J)a_i \alpha_j \kappa([m] \backslash \sigma_i, \sigma_i) \frown \kappa(J \backslash \tau_j, \tau_j)^*\right]\\
&= \left[ \sum_{\hat{i},\hat{j}} A_{\hat{i},\hat{j}} \kappa([m] \backslash \sigma_{\hat{i}}, \sigma_{\hat{i}}) \frown \kappa(J \backslash \tau_{\hat{j}}, \tau_{\hat{j}})^*\right]\\
&= \left[ \sum_{\hat{i},\hat{j}} A_{\hat{i},\hat{j}} \kappa(([m] \backslash \sigma_{\hat{i}}) \backslash (J \backslash \tau_{\hat{j}}), \sigma_{\hat{i}} \backslash \tau_j) \right]\\
&= \left[ \sum_{\hat{i},\hat{j}} A_{\hat{i},\hat{j}} \kappa(([m] \backslash J) \backslash (\sigma_{\hat{i}} \backslash \tau_{\hat{j}}), \sigma_{\hat{i}} \backslash \tau_j) \right] \in H_{n+m-(l+|J|+1)}(\ZK)
\end{aligned}
\]

\noindent where $\hat{i},\hat{j}$ are the pairs for which the cap product $\kappa([m] \backslash \sigma_{\hat{i}}, \sigma_{\hat{i}}) \frown \kappa(J \backslash \tau_{\hat{j}}, \tau_{\hat{j}})^*$ is non-trivial, and $A_{\hat{i},\hat{j}} = \mathrm{sgn}(\tau_j,J)a_i \alpha_j \neq 0 $. The last equality follows since both $J$ and $\sigma_i$ contain $\tau_j$, and $J \cap \sigma_i = \tau_j$ since the cap product is non-trivial.

The image of $[\mu] \frown h^*([\tau])$ under the inverse of the homology isomorphism of \eqref{eq:hochster} is
\[
\begin{aligned}
(h_*)^{-1}([\mu] \frown h_c([\tau]) = \left[ \sum_{\hat{i},\hat{j}} A_{\hat{i},\hat{j}} \sigma_{\hat{i}} \backslash \tau_{\hat{j}} \right] \in \tilde{H}_{n-l-2}(\cK_{[m] \backslash J}).
\end{aligned}
\]

We therefore have that under the composition of isomorphisms
$$(h_h)^{-1} \circ ([\mu] \frown (-)) \circ h_c\colon\bigoplus_{J \subseteq [m]} \tilde{H}^{l-|J| -1} (\cK_J) \rightarrow \bigoplus_{L \subseteq [m]} \tilde{H}_{(n+m-l)-|L|-1}(\cK_L)$$ the image of each of the groups $ \tilde{H}^{l} (\cK_J)$ is contained in $\tilde{H}_{n-l-2}(\cK_{[m] \backslash J})$. These groups are therefore isomorphic, and $\cK$ therefore satisfies $(n-1)$-dimensional combinatorial Alexander duality.

\end{proof}

We obtain a relation between Poincar\'e duality of moment-angle complexes and combinatorial Alexander duality of simplicial complexes, and conclude that there are no Poincar\'e duality moment-angle complexes which are not manifolds.

\begin{corollary} 
\label{cor:ZKPDKFWZKmdf}
Let $\cK$ be a simplicial complex on $[m]$ with non-trivial cohomology. Then the following are equivalent:
\begin{enumerate}[i)]
    \item $\ZK$ is an $(n+m)$-Poincar\'e duality space over $\Z$
    \item $\cK$ has $(n-1)$-dimensional combinatorial Alexander duality
    \item $\ZK$ is an $(n+m)$-dimension manifold.

\end{enumerate}
\qed
\end{corollary}

\subsection{Gorenstein duality of $\Z[\cK]$}
\label{subsec:gorensteindualityinZK}

A Noetherian ring satisfies Gorenstein duality if its localisation at every maximal ideal exhibits a certain form of self duality. In this paper, we relate the Gorenstein duality of Stanley-Reisner rings of simplicial complexes to Poincare dulaity of moment-angel complexes $\ZK$.

Let $\cK$ be a simplicial complex on vertex set $[m]$, and $R$ a commutative ring. The Stanley-Reisner ring is

\[R[\cK] =R[v_1,...,v_m] / \cI_{\cK}\]
where
\[\cI_{\cK}  =(v_{i_1} ... v_{i_j} \; | \; \{i_1,...,i_j\} \notin \cK )\]
is the Stanley-Reisner ideal, that is, the ideal generated by monomials corresponding to missing faces of $\cK$. 

\goodbreak

By a result of Stanley \cite[Theorem~5.1]{alma991027115349703276}, the Stanley-Reisner ring $\Z[\cK]$ having Gorenstein duality is equivalent to $\cK^*$ being an integral generalised homology $d$-sphere, where $\cK^* = \cK_{\{v \in [m] \; | st_{\cK}(v) \neq \cK\}}$ is the core of $\cK$, and $d$ is the dimension of $\cK^*$.

Notice that if $\cK$ has non-trivial cohomology, then $\cK = \cK^*$. Theorem~\ref{thm:ZKisPDiffKGHS} together with Stanley's \cite[Theorem~5.1]{alma991027115349703276} relates Poincar\'e duality of moment-angle complexes, Gorenstein duality of Stanley-Reisner rings, and combinatorial Alexander duality of simplicial complexes. We obtain an interplay between algebraic, combinatorial and topological dualities.

\goodbreak

\begin{theorem}
\label{thm:ZKPDiffKFWiffZ[K]Gor}
Let $\cK$ be a simplicial complex on $[m]$ with non-trivial cohomology. The following are equivalent:

\begin{enumerate}[i)]
    \item $\ZK$ is an $(n+m)$-Poincar\'e duality space
    \item $\cK$ has $(n-1)$ dimensional combinatorial Alexander duality
    \item $\Z[\cK]$ has Gorenstein duality.
\end{enumerate}
\qed
\end{theorem}

\goodbreak

\subsection{The polyhedral join product}

We extend our characterisation of Poincar\'e duality in $\ZK$ by utilising the polyhedral join product of simplicial complexes. 

\begin{definition} 
\label{def:polyjoinprod}
Let $\cK$ be a simplicial complex on $[m]$, and for $1 \leq i \leq m$, let $(\cK_i,\cL_i)$ be a simplicial pair on $[l_i]$, where the sets $[l_i]$ are pairwise disjoint. The \textit{polyhedral join product} is the simplicial complex on vertex set $[l_1] \sqcup ... \sqcup [l_m]$, defined as
 
 \[(\mathbfcal{K}_{\mathbf{i}},\mathbfcal{L}_{\mathbf{i}})^{*\cK} = \bigcup_{\sigma \in \cK} (\mathbfcal{K}_{\mathbf{i}},\mathbfcal{L}_{\mathbf{i}})^{*\sigma} 
\text{ where } 
(\mathbfcal{K}_{\mathbf{i}},\mathbfcal{L}_{\mathbf{i}})^{*\sigma} = \biggast_{i=1}^m \mathcal{Y}_i, \quad \mathcal{Y}_i = \begin{cases} \cK_i & i \in \sigma \\ \cL_i & \text{otherwise}.\end{cases}\]

\end{definition}

Let $l = \sum_{i=1}^m l_i$ where $l_i \geq 1 \; \forall i$, and let $(\mathbf{X},\mathbf{A})$ be an $l$-tuple of CW complexes, partitioned into $m$ distinct $l_i$ tuples with $\mathbf{X_i} = \{ {X_i}_j\}_{j=1}^{l_i}$ and $\mathbf{A_i} = \{ {A_i}_j\}_{j=1}^{l_i}$. It was proven by Vidaurre {\cite[Theorem~2.9]{Vidaurre}} that the polyhedral join product and the polyhedral product interact in the following way
\begin{equation} \label{eq:polyprodandpolyjoin}
(\mathbf{X},\mathbf{A})^{(\mathbfcal{K}_{\mathbf{i}},\mathbfcal{L}_{\mathbf{i}})^{*\cK} }= \left(\mathbf{(X_i,A_i)^{\mathbfcal{K}_i}},\mathbf{(X_i,A_i)^{\mathbfcal{L}_i}} \right)^{\cK}.
\end{equation}

We make use of this fact in extending our classification of Poincar\'e duality moment-angle complexes to polyhedral products whose entries are themselves moment-angle complexes.

\begin{proposition}
\label{thm:PDFanWangGorensteinrelationshipfullgenerality}

Let $\cK$ be a simplicial complex on $[m]$, and for $1 \leq i \leq m$, $(\cK_i,\cL_i)$ a simplicial pair on $[l_i]$. Suppose that the polyhedral join $(\mathbfcal{K}_{\mathbf{i}},\mathbfcal{L}_{\mathbf{i}})^{*\cK}$ has non-trivial cohomology. Then the following are equivalent.

\begin{enumerate}[i)]
\item The polyhedral product $(\mathbfcal{Z}_{{\mathbfcal{K}}_{\mathbf{i}}}, \mathbfcal{Z}_{{\mathbfcal{L}}_{\mathbf{i}}})^{\cK}$ is a Poincar\'e duality space.
\item The polyhedral join product $(\mathbfcal{K}_{\mathbf{i}},\mathbfcal{L}_{\mathbf{i}})^{*\cK}$ has combinatorial Alexander duality.
\item The Stanley-Reisner ring $\Z[(\mathbfcal{K}_{\mathbf{i}},\mathbfcal{L}_{\mathbf{i}})^{*\cK})]$ has Gorenstein duality.
\end{enumerate}
\end{proposition}

\begin{proof}
The equivalence of i) and ii) follows from \eqref{eq:polyprodandpolyjoin} together with Theorem~\ref{thm:ZKPDiffKFWiffZ[K]Gor}. The equivalence of ii) and iii) follows from Theorem~\ref{thm:ZKPDiffKFWiffZ[K]Gor}.
\end{proof}
\goodbreak
\begin{example}
\begin{enumerate}[i)]
\item Let $\cK = \partial \Delta^1$, and let $(\cK_1,\cL_1) = (\cK_2,\cL_2) = (\begin{tikzpicture}[x=1cm, y=1cm]
\draw (0,0) node(1)[fill, circle, scale=0.3]{} ; 
\draw (0.3,0) node(3)[fill, circle, scale=0.3]{} ; 
\draw (0.15, 0.29) node(2)[fill, circle, scale=0.3]{};
\draw (1)--(2)--(3);
\end{tikzpicture}, \begin{tikzpicture}[x=1cm, y=1cm]
\draw (0,0) node(1)[fill, circle, scale=0.3]{} ; 
\draw (0.3,0) node(3)[fill, circle, scale=0.3]{} ; 
\draw (0.15, 0.29) node(2)[draw=black,fill=white, circle, scale=0.3]{};
\end{tikzpicture})$, where $\circ$ denotes a ghost vertex. Then $(\mathbfcal{K}_{\mathbf{i}},\mathbfcal{L}_{\mathbf{i}})^{*\cK}$ is a 6-vertex triangulation of $S^2$, and in particular is a generalised homology sphere, so that 
\[(\mathbfcal{Z}_{{\mathbfcal{K}}_{\mathbf{i}}}, \mathbfcal{Z}_{{\mathbfcal{L}}_{\mathbf{i}}})^{\cK} = (S^3 \times D^2, S^3 \times S^1)^{\partial \Delta^1}\]
is a Poincar\'e duality space. Indeed, applying \eqref{eq:polyprodandpolyjoin}, and realising this 6-vertex triangulation of $S^2$ as $\partial \Delta^1 * \partial \Delta^1 * \partial \Delta^1 $, we have
\[(S^3 \times D^2, S^3 \times S^1)^{\partial \Delta^1} =  \cZ_{\partial \Delta^1} \times \cZ_{\partial \Delta^1} \times  \cZ_{\partial \Delta^1} = S^3 \times S^3 \times S^3.\]

\item Generalising the previous example, let $\cK = \partial \Delta^1$, and let $(\cK_1,\cL_1) = (\cK_2,\cL_2) = (\partial \Delta^{n-1} * \{v\}, \partial \Delta^n * \{\circ\})$, where $\circ$ denotes a ghost vertex. Then $(\mathbfcal{K}_{\mathbf{i}},\mathbfcal{L}_{\mathbf{i}})^{*\cK}$ is a $(2n+2)$-vertex triangulation of the $(n+1)$-sphere, and thus
\[(\mathbfcal{Z}_{{\mathbfcal{K}}_{\mathbf{i}}}, \mathbfcal{Z}_{{\mathbfcal{L}}_{\mathbf{i}}})^{\cK} = (S^{2n-1} \times D^2, S^{2n-1} \times S^1)^{\partial \Delta^1} \]
is a Poincar\'e duality space. 
\item Let $\cK = \partial \Delta^1 * \{j\}$, $(\cK_1,\cL_1) = (\cK_2,\cL_2) = (\{v\},\{\emptyset\})$,  and $(\cK_j,\cL_j) = (\begin{tikzpicture}[x=1cm, y=1cm]
\draw (0,0) node(1)[fill, circle, scale=0.3]{} ; 
\draw (0.3, 0) node(2)[fill, circle, scale=0.3]{};
\draw (0.4, 0.28) node(3)[fill, circle, scale=0.3]{};
\draw (0.15, 0.42) node(4)[fill, circle, scale=0.3]{};
\draw (-0.1, 0.28) node(5)[fill, circle, scale=0.3]{};
\draw (1)--(2)--(3)--(4)--(5)--(1);
\end{tikzpicture}, \emptyset)$. Then  $(\mathbfcal{K}_{\mathbf{i}},\mathbfcal{L}_{\mathbf{i}})^{*\cK}$ is a $7$-vertex triangulation of $S^2$. Here, $(\cZ_{\cK_1},\cZ_{\cL_1})  = (S^3,T^2)$ and $(\cZ_{\cK_2},\cZ_{\cL_2}) = ((S^3 \times S^4)^{\# 5},T^5) $, so that
\[\left((S^3,T^2), ((S^3 \times S^4)^{\# 5},T^5)\right)^{\Delta^1} \]
is a Poincar\'e duality space.
\end{enumerate}
\end{example}

These examples demonstrate that there are a variety of polyhedral join products which give rise to Poincar\'e duality spaces $\ZK$. The classification of polyhedral join products which are generalised homology spheres in the special case of composition complexes enables us to extend our results on duality.
Recall that composition complexes are the special case of the polyhedral join product where for all $i$, $\cK_i = \Delta^{n_i}$, $n_i \geq 1$. Ayzenberg~\cite[Theorem~6.6]{anton2013composition} proved that the composition complex $\cK(\cK_1,...,\cK_m)=(\Delta^{n_i}, \cK_i)^{*\cK}$ is a generalised homology sphere if and only if $\cK$ is a generalised homology sphere, for any non-ghost vertex $i$ of $\cK$, $\cK_i = \partial \Delta^{l_i -1}$, and for any ghost vertex $i$ of $\cK$, $\cK_i$ is a generalised homology sphere. 

Utilising this result together with \eqref{eq:polyprodandpolyjoin}, and the fact that the polyhedral product is a homotopy functor \cite[Proposition~8.1.1]{buchstaber2014toric}, we obtain the following corollary.

\begin{corollary}
\label{cor:charofwhenzkonsimpcomppoincareduality}

Let $\cK$ be a complex on $[m]$ with no ghost vertices, and let $\cK_1,...,\cK_m$ be complexes on $[l_1],...,[l_m]$, respectively. Then, $(\mathbf{C}{\mathbfcal{Z}_{\mathbfcal{K}}}_i, {\mathbfcal{Z}_{\mathbfcal{K}}}_i )^{\cK}$ is a Poincar\'e duality space if and only if $\cK$ is a generalised homology sphere, and for all $i$, $\cK_i = \partial \Delta^{[l_i]}$. 
\qed
\end{corollary}

\goodbreak

\begin{example} 
\begin{enumerate}[i)]
\item
 Let $\cK$ be a simplicial complex on $[m]$. For $1 \leq i \leq m$, let $l_i \geq 2$.
    Then $(CT^{l_i},T^{l_i})^{\cK}$ is a Poincar\'e duality space if and only if $\cK$ consists solely of ghost vertices. 
    \item Let $\cK$ be a simplicial complex on $[m]$, and for $1 \leq i \leq m$, let $l_i \geq 1$. Then $(D^{2l_i},S^{2l_i -1})^{\cK}$ is a Poincar\'e duality space if and only if $\cK$ is a generalised homology sphere. 
    \end{enumerate}
\end{example}

\bibliographystyle{abbrv} 
\bibliography{bibliography}

\end{document}